\documentclass[final]{siamltex1213}

\usepackage[english]{babel}
\usepackage[utf8]{inputenc}
\usepackage{amsmath}
\usepackage{amsfonts}
\usepackage{program,color}
\usepackage[square,numbers]{natbib}
\usepackage{algorithmicx}
\usepackage[ruled]{algorithm}
\usepackage{algpseudocode}
\newcommand{\specialcell}[2][c]{\begin{tabular}[#1]{@{}c@{}}#2\end{tabular}}

\newtheorem{fact}{Fact}
\newtheorem{remark}{Remark}

\newcommand\given[1][]{\:#1\vert\:}
\newcommand{\half}{\tfrac1{2}}
\newcommand{\E}{\mathbb{E}}

\newcommand{\bX}{{\boldsymbol{X}}}
\newcommand{\Id}{{\boldsymbol{Id}}}
\newcommand{\bSigma}{{\boldsymbol{\Sigma}}}
\newcommand{\by}{{\boldsymbol{y}}}
\newcommand{\bz}{{\boldsymbol{z}}}
\newcommand{\be}{{\boldsymbol{e}}}
\newcommand{\bw}{{\boldsymbol{w}}}
\newcommand{\bres}{{\boldsymbol{r}}}
\newcommand{\res}{{{r}}}
\newcommand{\maxIter}{{{maxIter}}}
\newcommand{\bbeta}{{\boldsymbol{\beta}}}
\newcommand{\balpha}{{\boldsymbol{\alpha}}}
\newcommand{\bgamma}{{\boldsymbol{\gamma}}}
\newcommand{\bP}{{\boldsymbol{P}}}
\newcommand{\betaln}{{\boldsymbol{\beta_{LN}}}}
\newcommand{\betals}{{\boldsymbol{\beta_{LS}}}}
\newcommand{\betaopt}{{\boldsymbol{\beta^{\star}}}}
\title{Convergence properties of the randomized extended Gauss-Seidel and Kaczmarz methods} 
\author{Anna Ma, Deanna Needell, Aaditya Ramdas} 
\date{\today}

\begin{document}
\maketitle

\begin{abstract}
The Kaczmarz and Gauss-Seidel methods both solve a linear system $\bX\bbeta = \by$ by iteratively refining the solution estimate.  Recent interest in these methods has been sparked by a proof of Strohmer and Vershynin which shows the \textit{randomized} Kaczmarz method converges linearly in expectation to the solution. Lewis and Leventhal then proved a similar result for the randomized Gauss-Seidel algorithm.  However, the behavior of both methods depends heavily on whether the system is under or overdetermined, and whether it is consistent or not.  Here we provide a unified theory of both methods, their variants for these different settings, and draw connections between both approaches.  In doing so, we also provide a proof that an extended version of randomized Gauss-Seidel converges linearly to the least norm solution in the underdetermined case (where the usual randomized Gauss Seidel fails to converge).  We detail analytically and empirically the convergence properties of both methods and their extended variants in all possible system settings.  With this result, a complete and rigorous theory of both methods is furnished. 
\end{abstract}

\section{Introduction}
We consider solving a linear system of equations 
\begin{equation}\label{eq:syseq}
\bX \bbeta = \by,
\end{equation}
for a (real or complex) $m\times n$ matrix $\bX$, in various problem settings.  Recent interest in the topic was reignited when \citet{SV09:Randomized-Kaczmarz} proved the linear\footnote{Mathematicians often refer to linear convergence as exponential convergence.} convergence rate of the Randomized Kaczmarz (RK) algorithm that works on the rows of $\bX$ (data points). Following that, \citet{LL10:Randomized-Methods} proved the linear convergence of a Randomized Gauss-Seidel (RGS), i.e. Randomized Coordinate Descent, algorithm that works on the columns of $\bX$ (features).

When the system of equations is inconsistent (i.e. has no exact solution), as is typically the case when $m>n$ in real-world overconstrained systems, RK is known to not converge to the ordinary least squares solution 
\begin{equation}\label{eq:LS}
\betals := \arg\min_{\bbeta} \half \|\by - \bX \bbeta\|_2^2 
\end{equation}
as studied by \citet{Nee10:Randomized-Kaczmarz}. \citet{ZF12:Randomized-Extended} extended the RK method with the modified Randomized Extended Kaczmarz (REK) algorithm, which linearly converges to $\betals$. Interestingly, in this setting, we will argue in Section \ref{sec:undercomplete} that RGS does converge to $\betals$ without any special extensions.

\subsection{Motivation and contribution}

The above introduction represents only half the story. When $m<n$, there are fewer constraints than variables, and the system has infinitely many solutions. In this case, especially if we have no prior reason to believe any additional sparsity in the signal structure, we are often interested in finding the least Euclidean norm solution:
\begin{equation}\label{eq:LN}
\betaln := \arg\min_{\bbeta} \|\bbeta\|_2 \mbox{~ s.t. ~} \by = \bX\bbeta.
\end{equation}
While RGS converges to $\betals$ in the overcomplete setting, we shall argue in Section \ref{sec:undercomplete}  that in the undercomplete setting it does not converge to $\betaln$. We will also argue that RK does converge to $\betaln$ without any extensions in this setting.

The main contribution of our paper is to provide a unified theory of these related iterative methods. We will also construct an extension to RGS that parallels REK, which unlike RGS does converge to $\betaln$ (just as REK, unlike RK, converges to $\betals$). Some desired properties for this algorithm include that it should \textit{also} converge linearly, not require much extra computation, and work well in simulations. We shall see that our Randomized Extended Gauss-Seidel (REGS) method does indeed possess these desired properties. A summary of this unified theory is provided in Table~\ref{tab1}.

\begin{table}[ht]
\begin{tabular}{|l|l|l|l|}
\hline
Method & \specialcell{Overconstrained,\\ consistent : \\ convergence to $\betaopt$? } & \specialcell{Overconstrained,\\ inconsistent : \\ convergence to $\betals$?} & \specialcell{Underconstrained :\\ convergence to $\betaln$?} \\
\hline
RK & Yes \cite{SV09:Randomized-Kaczmarz} & No \cite[Thm. 2.1]{Nee10:Randomized-Kaczmarz} & Yes (Sec. \ref{sec:undercomplete})\\
\hline
REK & Yes \cite{ZF12:Randomized-Extended} & Yes \cite{ZF12:Randomized-Extended} & Yes (Sec. \ref{sec:undercomplete})\\
\hline
RGS & Yes \cite{LL10:Randomized-Methods} & Yes \cite{LL10:Randomized-Methods} & No (Sec. \ref{sec:undercomplete})\\
\hline
REGS & Yes (Remark \ref{regsoc}) & Yes (Sec. \ref{sec43}) & Yes (Thm. \ref{thm:main}) \\
\hline
\end{tabular}
\caption{Summary of convergence properties for the overdetermined and consistent setting, overdetermined and inconsistent setting, and underdetermined settings.  We write $\betaopt$ to denote the solution to \eqref{eq:syseq} in the overdetermined consistent setting, with $\betals$ and $\betaln$ being defined in \eqref{eq:LS} and \eqref{eq:LN} for the other two settings. 
}
\label{tab1}
\end{table}

\subsection{Paper Outline}

In Section~\ref{sec:algs} we recap the three main existing algorithms mentioned in the introduction (RK, RGS, REK).  We discuss the performance of these algorithms in the three natural settings described in Table~\ref{tab1} in Section~\ref{sec:variants}. Section~\ref{sec:REGS} introduces our proposed algorithm (REGS) and proves its linear convergence to the least norm solution, completing the theoretical framework.  Lastly, we end with some simulation experiments in Section~\ref{sec:exp} to demonstrate the tightness and usefulness of our theory, and conclude in Section~\ref{sec:conclude}.

\section{Existing Algorithms and Related Work}\label{sec:algs}

In this section, we will summarize the algorithms mentioned in the introduction, i.e. RK, RGS and REK. We will describe their iterative update rules and mention their convergence guarantees, leaving the details of convergence to the next section.  Throughout the paper we will use the notation $\bX^i$ to represent the $i$th row of $\bX$ (or $i$th entry in the case of a vector) and $\bX_{(j)}$ to denote the $j$th column of a matrix $\bX$.  We will write the estimation $\bbeta$ as a column vector.  We write vectors and matrices in boldface, and constants in standard font.

\subsection{Randomized Kaczmarz (RK)}

The Kaczmarz method was first introduced in the notable work of \citet{Kac37:Angenaeherte-Aufloesung}.  It has gained recent interest in tomography research where it is known as the \textit{Algebraic Reconstruction Technique} (ART) \cite{GBH70:Algebraic-Reconstruction,Nat01:Mathematics-Computerized,Byr08:Applied-Iterative,herman2009fundamentals}.  Although in its original form the method selects rows in a deterministic fashion (often simply cyclically), it has been well observed that a random selection scheme reduces the possibility of a poor choice of row ordering \cite{HS78:Angles-Null,HM93:Algebraic-Reconstruction}.  Earlier convergence analysis of the randomized variant were obtained (e.g. \cite{XZ02:Method-Alternating}), but yielded bounds with expressions that were difficult to evaluate.  \citet{SV09:Randomized-Kaczmarz} showed that the RK method described above 
has an expected linear convergence rate to the solution $\betaopt$ of \eqref{eq:syseq}, and are the first to provide an explicit convergence rate in expectation which depends only on the geometric properties of the system.  This work was extended by \citet{Nee10:Randomized-Kaczmarz} to the inconsistent case, analyzed almost surely by \citet{CP12:Almost-Sure-Convergence}, accelerated in several ways \cite{Elf80:Block-Iterative-Methods,EN11:Acceleration-Randomized,popa2012kaczmarz,NW12:Two-Subspace-Projection,needell2013paved}, and extended to more general settings \cite{LL10:Randomized-Methods,richtarik2012iteration,NSWjournal}. 

We describe here the randomized variant of the Kaczmarz method put forth by \citet{SV09:Randomized-Kaczmarz}.
Taking  $\bX, \by$ as input and starting from an arbitrary initial estimate for $\bbeta$ (for example $\bbeta_0 = \bf{0}$), RK repeats the following in each iteration.  First, a random row $i \in \{1 ,..., m\}$ is selected with probability proportional to its Euclidean norm, i.e.
$$
\Pr(\text{row} = i) = \frac{\|\bX^i\|^2_2}{\|\bX\|_F^2},
$$
where $\|\bX\|_F$ denotes the Frobenius norm of $\bX$.
Then, project the current iterate onto that row, i.e.
\begin{equation}
\bbeta_{t+1} := \bbeta_t + \frac{(y^i - \bX^{i}\bbeta_t)}{\|\bX^i\|^2_2} (\bX^i)^*,
\end{equation}
where here and throughout $\bX^*$ denotes the (conjugate) transpose of $\bX$.

Intuitively, this update can be seen as greedily satisfying the $i$th equation in the linear system.  Indeed, it is easy to see that after the update,
\begin{equation}
\bX^{i}\bbeta_{t+1} = y^i.
\end{equation}

Referring to \eqref{eq:LS} and defining
$$
L(\bbeta) = \half \|\by-\bX\bbeta\|^2 = \half \sum_{i=1}^m (y^i - \bX^{i} \bbeta)^2,
$$
 we can alternatively interpret this update as stochastic gradient descent (choosing a random data-point on which to update), where the step size is the inverse Lipschitz constant of the stochastic gradient  
$$
\nabla^2 \half (y^i - \bX^{i} \bbeta)^2 = \|\bX^i\|^2_2.
$$

\subsection{Randomized Extended Kaczmarz (REK)}

For inconsistent systems, the RK method does not converge to the least-squares solution as one might desire.  This fact is clear since the method at each iteration projects completely onto a selected solution space, being unable to break the so-called \textit{convergence horizon}.  One approach to overcome this is to use relaxation parameters, so that the estimates are not projected completely onto the subspace at each iteration \cite{whitney1967two,tanabe1971projection,censor1983strong,hanke1990acceleration}. Recently, Zouzias and Freris \cite{ZF12:Randomized-Extended} proposed a variant of the RK method motivated by the work of \citet{Pop98:Extensions-Block-Projections} which instead includes a random projection to iteratively reduce the component of $\by$ which is orthogonal to the range of $\bX$.  This method, named Randomized Extended Kaczmarz (REK) can be described by the following iteration updates, which can be initialized with $\bbeta_0 = \boldsymbol{0}$ and $\bz_0 = \by$:
\begin{equation}
\bbeta_{t+1} := \bbeta_t + \frac{(y^i - z^i_t- \bX^{i}\bbeta_t)}{\|\bX^i\|^2_2} (\bX^i)^*, \quad \bz_{t+1} = \bz_t - \frac{\langle \bX_{(j)}, \bz_t \rangle}{\|\bX_{(j)}\|_2^2}\bX_{(j)}.
\end{equation}
Here, a column $j \in \{1,..., n\}$ is also selected at random with probability proportional to its Euclidean norm:
\begin{equation}\label{eq:c}
\Pr(\text{column} = j) = \frac{\|\bX_{(j)}\|^2_2}{\|\bX\|_F^2},
\end{equation}
 and again $\bX_{(j)}$ denotes the $j$th column of $\bX$.  Here, $\bz_t$ approximates the component of $\by$ which is orthogonal to the range of $\bX$, allowing for the iterates $\bbeta_t$ to converge to the true least-squares solution of the system.  \citet{ZF12:Randomized-Extended} prove that REK converges linearly in expectation to this solution $\betals$.

\subsection{Randomized Gauss-Seidel (RGS)}

Again taking $\bX, \by$ as input and starting from an arbitrary $\bbeta_0$, the Randomized Gauss-Seidel (RGS) method (or the Randomized Coordinate Descent method) repeats the following in each iteration. First, a random column $j \in \{1,..., n\}$ is selected as in \eqref{eq:c}.
We then minimize the objective $L(\bbeta) = \half \|\by-\bX\bbeta\|^2_2$ with respect to this coordinate to get
\begin{equation}\label{eq:rgs}
\bbeta_{t+1} := \bbeta_t + \frac{\bX_{(j)}^*(\by-\bX\bbeta_t)}{\|\bX_{(j)}\|^2_2} \be_{(j)}
\end{equation}
where $\be_{(j)}$ is the $j$th coordinate basis column vector (all zeros with a $1$ in the $j$th position). It can be seen as greedily minimizing the objective with respect to the $j$th coordinate. Indeed, letting $\bX_{(-j)},\bbeta^{-j}$ represent $\bX$ without its $j$th column and $\bbeta$ without its $j$th coordinate,
\begin{equation}
\frac{\partial L}{\partial \bbeta^j} = -\bX_{(j)}^*(\by-\bX\bbeta) = -\bX_{(j)}^*(\by-\bX_{(-j)}\bbeta^{-j} - \bX_{(j)}\bbeta^j).
\end{equation}
Setting this equal to zero for the coordinate-wise minimization, we get the aforementioned update \eqref{eq:rgs} for $\bbeta^j$. Alternatively, since $[\nabla L(\bbeta)]^j = -\bX_{(j)}^*(\by-\bX\bbeta)$, the above update can intuitively be seen as a univariate descent step where the step size is the inverse Lipschitz constant of the gradient along the $j$th coordinate, since the $(j,j)$ entry of the Hessian is 
$$
[\nabla^2 L(\bbeta)]_{j,j} = (\bX^*\bX)_{j,j} = \|\bX_{(j)}\|^2_2.
$$

\citet{LL10:Randomized-Methods} showed that this algorithm has an expected linear convergence rate.  We will detail the convergence properties of this algorithm and the others in the next section.

\section{Problem Variations}\label{sec:variants}

We first examine the differences in behavior of the two algorithms RGS and RK in three distinct but related settings. This will highlight the opposite behaviors of these two similar algorithms.

When the system of equations \eqref{eq:syseq} has a unique solution, we represent this by $\betaopt$. This happens when $m \geq n$, and the system is  consistent. Assuming  that $\bX$ has full column rank, 
\begin{equation}
\betaopt = (\bX^*\bX)^{-1}\bX^*\by,
\end{equation}
and then $\bX\betaopt = \by$. 

When \eqref{eq:syseq} does not have any consistent solution, we refer to the least-squares solution of \eqref{eq:LS} as $\betals$. This could happen in the overconstrained case, when $m > n$. Again, assuming  that $\bX$ has full column rank, we have
\begin{equation}
\betals = (\bX^*\bX)^{-1}\bX^*\by,
\end{equation}
and we can write $\bres := \by - \bX\betals$ as the residual vector.

When \eqref{eq:syseq} has infinitely many solutions, we call the minimum Euclidean norm solution given by \eqref{eq:LN}, $\betaln$. This could happen in the underconstrained case, when $m < n$. Assuming that $\bX$ has full row rank, we have
\begin{equation}
\betaln = \bX^* (\bX\bX^*)^{-1} \by.
\end{equation}
In the above notation, 
the $LS$ stands for Least Squares and $LN$ for Least Norm. We shall return to each of these three situations in that order in future sections.

  One of our main contributions is to achieve a unified understanding of the behavior of RK and RGS in  these different situations. 
The literature for RK deals mainly with the first two settings only (see  \cite{SV09:Randomized-Kaczmarz}, \cite{Nee10:Randomized-Kaczmarz}, \cite{ZF12:Randomized-Extended}).  In the third setting, one readily obtains convergence to an \textit{arbitrary} solution (see e.g. (3) of \cite{liu2014asynchronous}), but the convergence to the least norm solution is not often studied (likely for practical reasons). 
The literature for RGS typically focuses on more general setups than our specific quadratic least squares loss function $L(\beta)$ (see \citet{nesterov2012efficiency} or \citet{richtarik2012iteration}). However, for both the purposes of completeness, and for a more thorough understanding of the relationship between RK and RGS, it turns out to be crucial to analyze all three settings (for equations \eqref{eq:syseq}-\eqref{eq:LN}).

\begin{enumerate}
\item When $\betaopt$ is a unique consistent solution, we present proofs of the linear convergence of both algorithms - the results are known from papers by \cite{SV09:Randomized-Kaczmarz} and \cite{LL10:Randomized-Methods} but are presented here in a novel manner so that their relationship becomes clearer and direct comparison is easily possible. 

\item When $\betals$ is the (inconsistent) least squares solution, we show why RGS iterates converge linearly to $\betals$, but RK iterates do not - making RGS preferable. These facts are not hard to see, but we make it more intuitively and mathematically clear why this should be the case.

\item When $\betaln$ is the minimum norm consistent solution, we explain why RK converges linearly to it, but RGS iterates do not (both so far seemingly undocumented observations) - making RK preferable. 
\end{enumerate}

Together, the above three points complete the picture (with solid accompanying intuition) of the opposing behavior of RK and RGS.  Later, we will present our variant of the RGS method, the Randomized Extended Gauss-Seidel (REGS), and compare with the corresponding variant of RK (REK).  This new analysis will complete the unified framework for these methods.

\subsection{Overconstrained System, Consistent}\label{sec31}

Here we will assume that $m > n$, $\bX$ has full column rank, and the system is consistent, so
 $\by = \bX\betaopt$.
First, let us write the updates used by both algorithms in a revealing fashion. If RK and RGS select row $i$ and column $j$ at step $t+1$, and $\be^i$ (resp. $\be_{(j)}$) is the $i$th coordinate basis row (resp. column) vector, then the updates can be rewritten as:
\begin{align}\label{eq:RKupdate}
&\mbox{(RK)}& \bbeta_{t+1} &:= \bbeta_t + \frac{\be^{i} \bres_t}{\|\bX^i\|^2_2} (\bX^i)^*&\\
&\mbox{(RGS)}& \bbeta_{t+1} &:= \bbeta_t + \frac{\bX_{(j)}^* \bres_t}{\|\bX_{(j)}\|^2_2}\be_j &
\end{align}
where $\bres_t = \by - \bX\bbeta_t = \bX\betaopt-\bX\bbeta_t$ is the residual vector at iteration $t$. Then multiplying both equations by $\bX$ gives
\begin{align}
&\mbox{(RK)}&\bX\bbeta_{t+1} &:= \bX\bbeta_t + \frac{\bX^{i}(\betaopt-\bbeta_t)}{\|\bX^i\|^2_2} \bX (\bX^i)^*&\\
&\mbox{(RGS)}& \bX\bbeta_{t+1} &:= \bX\bbeta_t + \frac{\bX_{(j)}^* \bX(\betaopt-\bbeta_t)}{\|\bX_{(j)}\|^2_2}\bX_{(j)}.& \label{eq:RGSupdate}
\end{align}

We now come to an important difference, which is the key update equation for RK and RGS. 

First, from the update (\ref{eq:RKupdate}) for RK, we have that $\bbeta_{t+1}-\bbeta_t$ is parallel to $\bX^i$.  Also, $\bbeta_{t+1}-\betaopt$ is orthogonal to $\bX^i$ (since $\bX^{i}(\bbeta_{t+1}-\betaopt)=y^i - y^i = 0$). Then by the Pythagorean theorem, 
\begin{equation}\label{eq:RKrecursion}
\|\bbeta_{t+1} - \betaopt\|^2_2 = \|\bbeta_t - \betaopt\|^2_2 - \|\bbeta_{t+1} - \bbeta_t\|^2_2.
\end{equation}
Note that from the update (\ref{eq:RGSupdate}), we have that $\bX\bbeta_{t+1} - \bX\bbeta_t$ is parallel to $\bX_{(j)}$. Also, $\bX\bbeta_{t+1} - \bX\betaopt$ is orthogonal to $\bX_{(j)}$ (since $\bX_{(j)}^*(\bX\bbeta_{t+1} - \bX\betaopt) = \bX_{(j)}^*(\bX\bbeta_{t+1} - \by) = 0$ by the optimality condition $\partial L/\partial \bbeta^{j} = 0$). Then again by the Pythagorean theorem,
\begin{equation}\label{eq:RGSrecursion}
\|\bX\bbeta_{t+1} - \bX\betaopt\|^2_2 = \|\bX\bbeta_t - \bX\betaopt\|^2_2 - \|\bX\bbeta_{t+1} - \bX\bbeta_t\|^2_2.
\end{equation}
The rest of the proof follows by simply substituting for the last term in the above two equations, and is presented in the following table for easy comparison. Note $\bSigma=\bX^*\bX$ is the full-rank covariance matrix and we first
 take expectations with respect to the randomness at the $(t+1)$st step, conditioning on all randomness up to the $t$th step. We later iterate this expectation.

\begin{table}[ht]
\begin{tabular}{|l|l|}
\hline
\specialcell{Randomized Kaczmarz:\\ $\quad\E_{t}\|\bbeta_{t+1} - \betaopt\|^2_2$} & \specialcell{Randomized Gauss-Seidel:\\ $\quad\E_{t}\|\bX\bbeta_{t+1} - \bX\betaopt\|^2_2$} \\
\hline  
$\begin{aligned}[t]
&= \|\bbeta_t - \betaopt\|^2_2 - \E\|\bbeta_{t+1} - \bbeta_t\|^2_2 \\
&= \| \bbeta_{t} - \betaopt\|^2_2 \\
& \quad -  \sum_i \frac{\|\bX^i\|^2_2}{\|\bX\|_F^2} \frac{(\bX^i(\bbeta_t - \betaopt))^2}{(\|\bX^i\|^2_2)^2} \|\bX^i\|^2_2  \\
&= \|\bbeta_{t} - \betaopt\|^2_2 \left(1 - \frac{\| \bX(\bbeta_t - \betaopt)\|^2_2}{\|\bX\|_F^2 \| \bbeta_{t} - \betaopt\|^2_2} \right) \\ 
&\leq \|\bbeta_t-\betaopt\|^2_2 \left(1 - \frac{\lambda_{\min}(\bSigma)}{Tr(\bSigma)}\right)  
\end{aligned}$
&
$\begin{aligned}[t]
&= \|\bX\bbeta_t - \bX\betaopt\|^2_2 - \E\|\bX\bbeta_{t+1} - \bX\bbeta_t\|^2_2\\
& = \|\bX \bbeta_{t} - \bX\betaopt\|^2_2 \\
& \quad - \sum_j \frac{\|\bX_{(j)}\|^2_2}{\|\bX\|_F^2} \frac{(\bX_{(j)}^* \bX(\bbeta_t - \betaopt))^2}{(\|\bX_{(j)}\|^2_2)^2}\|\bX_{(j)}\|^2_2\\
& = \|\bX \bbeta_{t} - \bX\betaopt\|^2_2 \left(1 -  \frac{\|\bX^* \bX(\bbeta_t - \betaopt)\|^2_2}{\|\bX\|_F^2\|\bX \bbeta_{t} - \bX\betaopt\|^2_2} \right) \\
& \leq \|\bX\bbeta_t-\bX\betaopt\|^2_2 \left(1 - \frac{\lambda_{\min}(\bSigma)}{Tr(\bSigma)}\right)
\end{aligned}$\\
\hline
\end{tabular}
\label{tab:over_compare}
\end{table}

Here, $\lambda_{\min}(\bSigma)\| \bbeta_t - \betaopt\|^2_2 \leq \|\bX(\bbeta_t - \betaopt)\|^2_2 $  i.e. $\lambda_{\min}(\bSigma)$ is the smallest eigenvalue of $\bSigma$ (singular value of $\bX$). 
It follows that
\begin{align}
&\mbox{(RK)}& \E\|\bbeta_t - \betaopt\|^2_2 &\leq \left( 1 - \frac{\lambda_{\min}(\bSigma)}{Tr(\bSigma)} \right)^{t}\|\bbeta_0-\betaopt\|^2_2& \label{eq:RGSlin}\\
&\mbox{(RGS)}& \E\|\bbeta_t - \betaopt\|_\bSigma^2 &\leq \left( 1 - \frac{\lambda_{\min}(\bSigma)}{Tr(\bSigma)} \right)^{t}\|\bbeta_0-\betaopt\|_\bSigma^2,\notag 
\end{align}
where $\|\bw\|_{\bSigma}^2 = \bw^*\bSigma\bw = \|\bX\bw\|_2^2$ is the norm induced by $\bSigma$.  
Since $\bSigma$ is invertible when $m>n$ and $\bX$ has full column rank, the last equation also implies linear convergence of $\E\|\bbeta_t-\betaopt\|_2^2$.
The final results exist in \citet{SV09:Randomized-Kaczmarz,LL10:Randomized-Methods} but there is utility in seeing the two proofs in a form that differs from their original presentation, side by side. In this setting, both RK and RGS are essentially equivalent (without computational considerations).

\subsection{Overconstrained System, Inconsistent}\label{sec32}

Here, we will assume that $m > n$, $\bX$ is full column rank, and the system is inconsistent, so $\by = \bX\betals + \bres$, where $\bres$ is such that $\bX^* \bres = 0$. It is easy to see this condition, because as mentioned earlier,
\[
\betals = (\bX^* \bX)^{-1}\bX^* \by,
\]
implying that $\bX^* \bX \betals = \bX^* \by$. Substituting $\by=\bX\betals+\bres$ gives that $\bX^*\bres=0$.

In this setting, RK is known to not converge to the least squares solution, as is easily verified experimentally and geometrically. The tightest convergence upper bounds known are by \cite{Nee10:Randomized-Kaczmarz} and \cite{ZF12:Randomized-Extended} who show that
\begin{align*}
\E\|\bbeta_t - \betals\|^2_2 &\leq \left( 1 - \frac{\lambda_{\min}(\bSigma)}{Tr(\bSigma)} \right)^{t}\|\bbeta_0-\betals\|^2_2 + \frac{\|\bres\|^2_2}{\lambda_{\min}(\bSigma)}\\
&= \left( 1 - \frac{\sigma^2_{\min}(\bX)}{\|\bX\|_F^2} \right)^{t}\|\bbeta_0-\betals\|^2_2 + \frac{\|\bres\|^2_2}{\sigma^2_{\min}(\bX)},
\end{align*}
where we write $\sigma_{\min}(\bX)$ to denote the smallest (non-zero) singular value of $\bX$ and again $\|\bX\|_F$ its Frobenius norm.
Attempting the previous proof, \eqref{eq:RKrecursion} no longer holds -- the Pythagorean theorem fails because $\bbeta_{t+1}-\betals$ is no longer orthogonal to $\bX^i$ since $\bX^{i}(\bbeta_{t+1}-\betals) = y^i-\bX^{i}\betals \neq 0 $.
Intuitively, the reason RK does not converge is that every update of RK (say of row $i$) is a projection onto the ``wrong'' hyperplane that has constant $y^i$ (where the ``right'' hyperplane would involve projecting onto a parallel hyperplane with constant $y^i-\res^i$ where $\bres$ was defined above). An alternative intuition is that all RK updates are in the span of the rows, but $\betals$ is not in the row span. These intuitive explanations are easily confirmed by experiments seen in \cite[pp. 787--788]{ZF12:Randomized-Extended},\cite[pp. 402]{Nee10:Randomized-Kaczmarz}.  \citet{ZF12:Randomized-Extended} alleviate this issue with the REK algorithm, whose convergence obeys
\begin{equation}\label{rekrate}
\E\|\bbeta_t - \betals\|^2_2 \leq \left( 1 - \frac{\sigma^2_{\min}(\bX)}{\|\bX\|_F^2}\right)^{\left \lfloor{t/2}\right \rfloor }\left(1 + 2\frac{\sigma^2_{\min}(\bX)}{\sigma^2_{\max}(\bX)}\|\betals\|_2^2 \right).
\end{equation}

However, the fate of RK doesn't hold for RGS. Almost magically, in the previous proof, the Pythagorean theorem still holds in equation \eqref{eq:RGSrecursion} because 
\begin{equation}
\label{eq:lem2incon}
\bX_{(j)}^*(\bX\bbeta_{t+1} - \bX\betals) = \bX_{(j)}^*(\bX\bbeta_{t+1} - \by) + \bX_{(j)}^*(\by-\bX\betals) = 0.
\end{equation}
The first term is 0 by the  optimality condition for $\bbeta_{t+1}$ i.e. $\bX_{(j)}^*(\bX\bbeta_{t+1} - \by) = \partial L/\partial \bbeta^j = 0$. The second term is zero by the global optimality of $\betals$, i.e. $\bX^*(\by - \bX\betals) = \nabla L = 0$. Also, $\bSigma$ is full rank as before.
Indeed, RGS works in the space of fitted values $\bX\bbeta$ and not the iterates $\bbeta$.

In summary, RK does not converge to the LS solution, but RGS does at the same linear rate. 
This is what motivated the development of Randomized Extended Kaczmarz (REK) by \citet{ZF12:Randomized-Extended} which, as discussed earlier, is a modification of RK designed to converge to $\betals$ by randomly projecting out $r$. 
An independent paper by \citet{frek} argues however that in this setting RGS is preferable to REK in terms of computational convergence.

\subsection{Underconstrained System, Infinite Solutions}\label{sec:undercomplete}

Here, we will assume that $m < n$, $\bX$ is full row rank and the system is consistent with infinitely many solutions. As mentioned earlier, it is easy to show that
\[
\betaln = \bX^*(\bX\bX^*)^{-1}\by
\]
(which clearly satisfies $\bX\betaln = \by$). Every other consistent solution can be expressed as 
$$
\bbeta = \betaln + \bz ~\mbox{~ where ~}~ \bX\bz=0.
$$ 
Clearly any such $\bbeta$ would also satisfy $\bX\bbeta = \bX\betaln=y$. Since $\bX\bz=0$, $\bz \perp \betaln$ implying $\|\bbeta\|^2 = \|\betaln\|^2 + \|\bz\|^2$, showing that $\betaln$ is indeed the minimum norm solution as claimed.

In this case, RK has good behavior, and starting from $\bbeta_0=0$, it does converge linearly to $\betaln$. Intuitively, $\betaln = \bX^* \balpha$ (for $\balpha = (\bX\bX^*)^{-1}\by$) and hence is in the row span of $\bX$. Starting from $\bbeta_0=0$, RK only adds multiples of rows to its iterates, and hence will never have any component orthogonal to the row span of $\bX$. There is exactly one solution with no component orthogonal to the row span of $\bX$, and that is $\betaln$, and hence RK converges linearly to the required point, where the rate can be bounded in exactly the same way as \eqref{eq:RGSlin}. It is important not to start from an arbitrary $\bbeta_0$ since the RK updates can never eliminate any component of $\bbeta_0$ that is perpendicular to the row span of $\bX$.  Of course, the same properties are shared by REK for this case as well. It is noted in \citet[Sec. 2.1]{ZF12:Randomized-Extended} that the REK converges at the same rate for underdetermined systems as it does overdetermined systems.

Mathematically, the previous earlier proof works because the Pythagorean theorem holds since it is a consistent system.  Now, $\bSigma$ is not full rank but note that since both $\betaln$ and $\bbeta_t$ are in the row span, $\bbeta_t - \betaln$ has no component orthogonal to $\bX$ (unless it equals zero in which case the algorithm has already converged). Hence $\lambda_{\min}(\bSigma)\| \bbeta_t - \betaln\|^2 \leq \|\bX(\bbeta_t - \betaln)\|^2$ holds, where $\lambda_{\min}(\bSigma)$ is now understood to be the smallest positive eigenvalue of $\bSigma$. To summarize, the exact same bound \eqref{eq:RGSlin} still holds in this case, with the appropriate understanding of $\lambda_{\min}(\bSigma)$ and under the assumption that the initialized $\bbeta_0$ is in the row span of $\bX$.

RGS unfortunately suffers the opposite fate. The iterates do not converge to $\betaln$, even though $\bX\bbeta_t$ does converge to $\bX\betaln$. Mathematically, the convergence proof still carries forward as before, but in the last step when $\bX^* \bX$ cannot be inverted because it is not full rank. Hence we get convergence of the residual to zero, without getting convergence of the iterates to the least norm solution.  Intuitively, the iterates of RGS add components to the estimates that are orthogonal to the row span of $\bX$. These components are never eliminated because in minimizing the residual, they are ignored. Therefore, RGS is able to minimize the residual without finding the least norm solution.

\textit{Unfortunately, when each update is cheaper for RK than RGS (due to matrix size), RGS is preferred for reasons of convergence and when it is cheaper for RGS than RK, RK is preferred.}

\section{REGS}\label{sec:REGS}

We next introduce an extension of RGS, analogous to the extension REK of RK.  The purpose of extending RK was to allow for convergence to the least squares solution.  Now, the purpose of extending RGS is to allow for convergence to the least norm solution.  We view this method as a completion to the unified analysis of these approaches, and it may also possess advantages in its own right.

\subsection{The algorithm}
Consider the linear system \eqref{eq:syseq} with $m < n$.
Let $\betaln$ denote the least norm solution of the underdetermined system as described in \eqref{eq:LN}.  The REGS algorithm is described by the following pseudo-code. Analogous to the role $\bz$ plays in REK, $\bz$ iteratively approximates the component in $\bbeta$ orthogonal to the row-span of $\bX$.  By iteratively removing this component, we converge to the least norm solution. Note that outputting $\bbeta_t$ instead of  $\bbeta_t^{LN} = \bbeta_t - \bz_t$ in Algorithm \ref{alg:regs} recovers the RGS algorithm. This may be preferable in the overdetermined setting.

\begin{algorithm}{}
	\caption{Randomized Extended Gauss-Seidel (REGS) }\label{alg:regs}
\begin{algorithmic}[1]
\Procedure{}{$\bX$, $\by$, \maxIter}\Comment{$m\times n$ matrix $\bX$, $\by\in\mathbb{C}^m$, maximum iterations $T$}
\State Initialize $\bbeta_{0} = \bf{0}$, $\bz_{0} = \bf{0}$
\For {$t=1,2,\ldots, T  $ }
	\State Choose column  $\bX_{(j)}$  with probability  $\frac{\|\bX_{(j)}\|_2^2}{\|\bX\|_F^2} $
	\State Choose row  $\bX^i$ with probability $\frac{\|\bX^i\|_2^2}{\|\bX\|_F^2}$
	\State Set $\bgamma_{t} = \frac{{\bX^*_{(j)}}(\bX\bbeta_{t-1} - \by)}{\|\bX_{(j)}\|^2_2}\be_{(j)} $
	\State Set $\bbeta_{t} = \bbeta_{t-1} + \bgamma_{t}$
	\State Set $\bP_i = \Id_n - \frac{(\bX^i)^*\bX^i}{\|\bX^i\|_2^2}$\Comment{$\Id_n$ denotes the $n\times n$ identity matrix}
	\State Update $\bz_t = \bP_i(\bz_{t-1} + \bgamma_t)$
	\State Update $\bbeta_t^{LN} = \bbeta_t - \bz_t$
\EndFor
\State Output $\bbeta_t^{LN}$
\EndProcedure
\end{algorithmic}
\end{algorithm}

\subsection{Main result}
Our main result for the REGS method shows linear convergence to the least norm solution.

\begin{theorem} The REGS algorithm outputs an estimate $\bbeta^{LN}_T$ such that 
\begin{equation} \mathbb{E}\| \bbeta^{LN}_T - \betaln \|^2_2 \leq \alpha^T \|\bbeta^{LN}\|^2_2 + 2\alpha^{\lfloor T/2 \rfloor}\frac{B}{1-\alpha}\label{rate:regs}
\end{equation}
where $B = \frac{\| \bX\betaln\|^2_2}{\|\bX\|^2_F}$ and $\alpha = \left(1 - \frac{\sigma^2_{min}(\bX)}{\|\bX\|^2_F} \right)$.
\label{thm:main}
\end{theorem}
\begin{proof}
We devote the remainder of this section to the proof of Theorem~\ref{thm:main}.

Let $\mathbb{E}_{t-1}$ denote the expected value conditional on the first $t-1$ iterations, and instate the notation of the theorem. That is, $\E_{t-1}[\cdot] = \E[\cdot \given i_1, j_1, i_2, j_2, ... i_{t-1}, j_{t-1}]$ where $i_{t^*}$ is the $t^{*th}$ row chosen and $j_{t^*}$ is the $t^{*th}$ column chosen. We denote conditional expectation with respect to the choice of column as $\E^j_{t-1}[\cdot] = \E[\cdot \given i_1, j_1, ... i_{t-1}, j_{t-1}, i_t]$. Similarly, we denote conditional expectation with respect to the choice of row as $\E^i_{t-1}[\cdot] = \E[\cdot \given i_1, j_1, ... i_{t-1}, j_{t-1}, j_t]$. Then note by the law of total expectation we have that $\E_{t-1}[\cdot] = \E^i_{t-1}[\E^j_{t-1}[\cdot]]$.  
We will use the following elementary facts and lemmas.

\begin{fact} (\cite[Fact 3]{ZF12:Randomized-Extended}) For any $\bP_i$ as in the algorithm, $\mathbb{E}\|\bP_i\bw\|^2_2 \leq \alpha\|\bw\|^2_2$ for any $\bw$ in the row span of $\bX$.
\label{fact:exp-pi}
\end{fact}

\begin{remark} In Algorithm \ref{alg:regs}, $\bbeta_t^{LN} - \betaln$ is in the row span of $\bX$ for $t \geq 1$. Clearly $\betaln \in range(\bX^T)$. To see that $\bbeta_t^{LN} \in range(\bX^T)$ note that $\bbeta_t^{LN}$ consists of the RGS approximation $\bbeta_t$ from which we subtract $\bz_t$. This removes the component of $\bbeta_t$ which is orthogonal to the row span of $\bX$. 
\label{rem1}
\end{remark}

\begin{lemma}(\cite[Thm. 3.6]{LL10:Randomized-Methods}) 
\label{lem:LL10}
We have that $$\mathbb{E}_{t-1}\|\bX\bbeta_t - \bX\betaln\|^2_2 \leq \alpha \|\bX\bbeta_{t-1} - \bX\betaln\|^2_2 $$ and that $$\mathbb{E}\|\bX\bbeta_t - \bX\betaln\|_2^2 \leq \alpha^t\|\bX\bbeta_0 - \bX\betaln\|^2_2.$$
\label{lem:RGS}
\end{lemma}

Now we first consider $\|\bbeta^{LN}_t - \betaln\|^2_2$:
\begin{align}
	\|\bbeta^{LN}_t - \betaln\|^2_2 &= \|\bbeta_t - \bz_t - \betaln \|^2_2 \notag \\
	&= \| \bbeta_t - \bP_i(\bz_{t-1} + \bgamma_t) - \bP_i\betaln - (\Id_n - \bP_i)\betaln  \|^2_2 \notag  \\
	&= \| \bbeta_t - \bP_i(\bz_{t-1} + \bbeta_t - \bbeta_{t-1}) - \bP_i\betaln - (\Id_n - \bP_i)\betaln  \|^2_2 \notag \\
	&= \| (\Id_n - \bP_i)\bbeta_t + \bP_i(\bbeta_{t-1}-\bz_{t-1}) - \bP_i\betaln - (\Id_n - \bP_i)\betaln  \|^2_2 \notag \\	
	&= \| (\Id_n - \bP_i)\bbeta_t + \bP_i\bbeta^{LN}_{t-1} - \bP_i\betaln - (\Id_n - \bP_i)\betaln  \|^2_2 \notag \\
	&= \| \bP_i(\bbeta^{LN}_{t-1} - \betaln) + (\Id_n - \bP_i)(\bbeta_t - \betaln) \|^2_2 \notag \\
	&= \|\bP_i(\bbeta^{LN}_{t-1} - \betaln) \|^2_2 +  \| (\Id_n - \bP_i)(\bbeta_t - \betaln) \|^2_2.
	\label{eq:firststep}
\end{align}
So far, we have only used substitution of variables as defined for the algorithm and that $\betaln = \bP_i\betaln + (\Id_n - \bP_i)\betaln$ is an orthogonal decomposition. We first focus on the expected value of the second term.

\begin{lemma} We also have that
$$\mathbb{E}_{t-1} \|(\Id_n - \bP_i)(\bbeta_{t} - \betaln)\|_2^2 \leq \frac{\alpha\|\bX(\bbeta_{t-1} - \betaln)\|_2^2}{\|\bX\|_F^2}.  $$
\label{lem:2}
\end{lemma}
\begin{proof}
{\allowdisplaybreaks
	\begin{align*}
		\mathbb{E}_{t-1} \|(\Id_n - \bP_i)(\bbeta_{t} &- \betaln)\|_2^2 \\
		&=\mathbb{E}_{t-1} [(\bbeta_{t} - \betaln)^* (\Id_n - \bP_i)^*(\Id_n - \bP_i) (\bbeta_{t} - \betaln)] \\
		&=\mathbb{E}_{t-1} [(\bbeta_{t} - \betaln)^* (\Id_n - \bP_i) (\bbeta_{t} - \betaln)] \\
		&=   \mathbb{E}_{t-1} \Bigg[(\bbeta_{t} - \betaln)^*\left(\frac{(\bX^i)^*\bX^i}{\|\bX^i\|^2_2} \right) (\bbeta_{t} - \betaln) \Bigg] \\
		&= \mathbb{E}_{t-1} \left[ \frac{\|\bX^i(\bbeta_{t} - \betaln)\|_2^2}{\|\bX^i\|_2^2} \right] \\
			&=   \E_{t-1}^j\Bigg[\mathbb{E}^i_{t-1}\frac{\|\bX^i(\bbeta_{t} - \betaln)\|_2^2}{\|\bX^i\|_2^2}\Bigg] \\
			&= \E_{t-1}^j\Bigg[\sum_{i=1}^m \frac{\|\bX^i(\bbeta_{t} - \betaln)\|_2^2}{\|\bX^i\|_2^2} \cdot \frac{\|\bX^i\|_2^2}{\|\bX\|^2_F} \Bigg]\\
			& =\E_{t-1}^j\Bigg[\frac{\|\bX(\bbeta_{t} - \betaln)\|_2^2}{\|\bX\|_F^2}\Bigg] \\
			& \leq\frac{\alpha\|\bX(\bbeta_{t-1} - \betaln)\|_2^2}{\|\bX\|_F^2}. 
			\end{align*}
			}

The first line follows by expanding the norm, the second line since $(\Id_n - \bP_i)$ is a projection matrix, the third line from the definition of $\bP_i$, the fourth line is computation, the fifth line follows from the law of total expectation, the next two lines are computation, and finally the last line follows by Lemma \ref{lem:RGS}. Notice that in the seventh line, $\E_{t-1}^j = \E_{t-1}$ because the random variable $\bbeta_t$ only depends on the choice of columns.
\end{proof}

We want to control the term $r_{t} = \mathbb{E}\| (\Id_n - \bP_i)(\bbeta_t - \betaln) \|^2_2$ by bounding it by some $\alpha$ and $B$ such that $r_t \leq \alpha^t B$. We calculate this here:
\begin{align*}
	\mathbb{E}\| (\Id_n - \bP_i)(\bbeta_t - \betaln) \|^2_2 &=\mathbb{E}[\mathbb{E}_{t-1}\| (\Id_n - \bP_i)(\bbeta_t - \betaln) \|^2_2] \\
	& \leq  \frac{\alpha\mathbb{E}\|\bX(\bbeta_{t-1} - \betaln)\|_2^2}{\|\bX\|_F^2}\\
	& \leq \alpha^{t} \frac{\|\bX\bbeta_0 - \bX\betaln\|_2^2}{\|\bX\|^2_F}.
\end{align*}
The first line follows by definition, the second is by Lemma \ref{lem:2}, and the third by Lemma \ref{lem:RGS}.
  
 Finally, we take the expected value of $\|\bbeta^{LN}_t - \betaln\|^2_2$. From equation (\ref{eq:firststep}) and using Fact \ref{fact:exp-pi} we obtain:
\begin{align*}
	\mathbb{E}\|\bbeta^{LN}_t - \betaln\|^2_2 &=  \mathbb{E}\|\bP_i(\bbeta^{LN}_{t-1} - \betaln) \|^2_2 + \mathbb{E}\| (\Id_n - \bP_i)(\bbeta_t - \betaln) \|^2_2 \\ 
	& \leq \alpha\mathbb{E}\|(\bbeta^{LN}_{t-1} - \betaln) \|^2_2 + \mathbb{E}\| (\Id_n - \bP_i)(\bbeta_t - \betaln) \|^2_2.
\end{align*}

We complete the proof using the following lemma from \cite{ZF12:Randomized-Extended}:
\begin{lemma} (\cite[Thm. 8]{ZF12:Randomized-Extended}) Suppose that for some $\alpha, \bar{\alpha} < 1$, the following bounds hold for all $t^* \geq 0$:
$$ \mathbb{E} \| \bbeta_{t^*}^{LN} - \betaln \|^2_2 \leq \alpha \mathbb{E} \| \bbeta_{t^*-1}^{LN} - \betaln\|^2_2 + r_{t^*} \text{ and } r_{t^*} \leq \bar{\alpha}^{t^*}B.$$
Then for any $T > 0$,
$$ \mathbb{E} \| \bbeta_{T}^{LN} - \betaln \|^2_2 \leq \alpha^T \|\bbeta_{0}^{LN} - \betaln\|^2_2 + (\alpha^{\lfloor T/2 \rfloor} + \bar{\alpha}^{\lfloor T/2 \rfloor}) \frac{B}{1-\alpha}.$$
\label{thm:rek_convergence}
\end{lemma}

Letting $\alpha = \bar{\alpha} = \alpha$, $r_t^* = \mathbb{E}\| (\Id_n - \bP_i)(\bbeta_t - \betaln) \|^2_2$, $B = \frac{\|\bX\bbeta_{0} - \bX\betaln\|_2^2}{\|\bX\|^2_F}$, and noting that $\bbeta_{0}^{LN}=\bbeta_{0}=0$, we complete the proof of Theorem \ref{thm:main}.

\end{proof}

\begin{remark}\label{regsoc}
Here we note that the same proof works for overdetermined systems.  In particular, this works because Lemma \ref{lem:LL10} holds for $\betals$ and $\bbeta^*$ also (see Thm. 3.6 in \cite{LL10:Randomized-Methods}). Also, Lemma \ref{lem:2} follows for both overdetermined consistent systems (see table in Section \ref{sec31}) as well as overdetermined inconsistent systems (from \eqref{eq:lem2incon} and subsequent arguments). 
\end{remark}

\subsection{Comparison}\label{sec43}

Theorem \ref{thm:main} shows that, like the RK and REK methods, REGS converges linearly to the least-norm solution in the underdetermined case.  We believe it serves to complement existing analysis and completes the theory of these iterative methods in all three cases of interest.  For that reason, we compare the three approaches for the underdetermined setting here.  For ease of comparison, set $\alpha$ as in Theorem \ref{thm:main}, and write $\kappa = \sigma_{\max}(\bX)/\sigma_{\min}(\bX)$ for the condition number of $\bX$.  From the convergence rate bounds for RK \cite{SV09:Randomized-Kaczmarz} and REK \cite{ZF12:Randomized-Extended} given in Section \ref{sec:variants}, and after applying elementary bounds to \eqref{rate:regs} of Theorem \ref{thm:main}, we have:

\begin{align}
&\text{(RK)}& \mathbb{E}\|\bbeta_t - \betaln\|_2^2 \quad  &\leq \quad  \alpha^t\|\betaln\|_2^2\label{comp:rk}\\
&\text{(REK)}& \mathbb{E}\|\bbeta_{2t} - \betaln\|_2^2 \quad  &\leq \quad  \alpha^t(1+2\kappa^2)\|\betaln\|_2^2\label{comp:rek}\\
&\text{(REGS)}& \mathbb{E}\|\bbeta_{2t} - \betaln\|_2^2 \quad  &\leq \quad  \alpha^t(1+2\kappa^2)\|\betaln\|_2^2\label{comp:regs}.
\end{align}

We find similar results in the overdetermined, inconsistent setting. Using the convergence rate bounds for RGS \cite{LL10:Randomized-Methods}, REK \cite{ZF12:Randomized-Extended}, and REGS (Theorem \ref{thm:main}), also given in section \ref{sec:variants}, we have:

\begin{align}
&\text{(RGS)}& \mathbb{E}\|\bbeta_t - \betals\|_2^2 \quad  &\leq \quad  \alpha^t\|\betals\|_2^2\label{comp:rk_oi}\\
&\text{(REK)}& \mathbb{E}\|\bbeta_{2t} - \betals\|_2^2 \quad  &\leq \quad  \alpha^t(1+2\kappa^2)\|\betals\|_2^2\label{comp:rek_oi}\\
&\text{(REGS)}& \mathbb{E}\|\bbeta_{2t} - \betals\|_2^2 \quad  &\leq \quad  \alpha^t(1+2\kappa^2)\|\betals\|_2^2\label{comp:regs_oi}.
\end{align}

Thus, up to constant terms (which are likely artifacts of the proofs), the bounds provide the same convergence rate $\alpha$, which is not surprising in light of the connections between the methods. In the next section, we compare these approaches experimentally.

\section{Empirical Results} \label{sec:exp}

In this section we present our experimental results. The code used to run these experiments can be found at \cite{code}. For each experiment, we initialize a matrix $\bX$ and vector $\bbeta$ with independent standard normal entries and run 50 trials. The right hand side $\by$ is taken to be $\bX\bbeta$. At each iteration $t$, we keep track of the $\ell_2$-error $\|\bbeta^{LN}_t - \betaln\|^2_2$ and fix the stopping criterion to be $\|\bbeta^{LN}_t - \betaln\|^2_2 < 10^{-6}$ (of course in practice one chooses a more practical criterion). In each plot, the solid blue line represents the median $\ell_2$-error at iteration $t$, the light blue shaded region captures the range of error across trials, and the red line represents the theoretical upper bound at each iteration. In Figure \ref{fig:vary_m}, we show the convergence of $\bbeta^{LN}_t$ for varying sized underdetermined linear systems. In Figure \ref{fig:fig2}, we show the convergence of a matrix $\bX$ of size 700x1000 and its theoretical upper bound. As it turns out, the REGS algorithm often converges much faster than the theoretical worst-case bound. 

\begin{figure}[ht]
                \includegraphics[natwidth=996,natheight=384, scale=0.45]{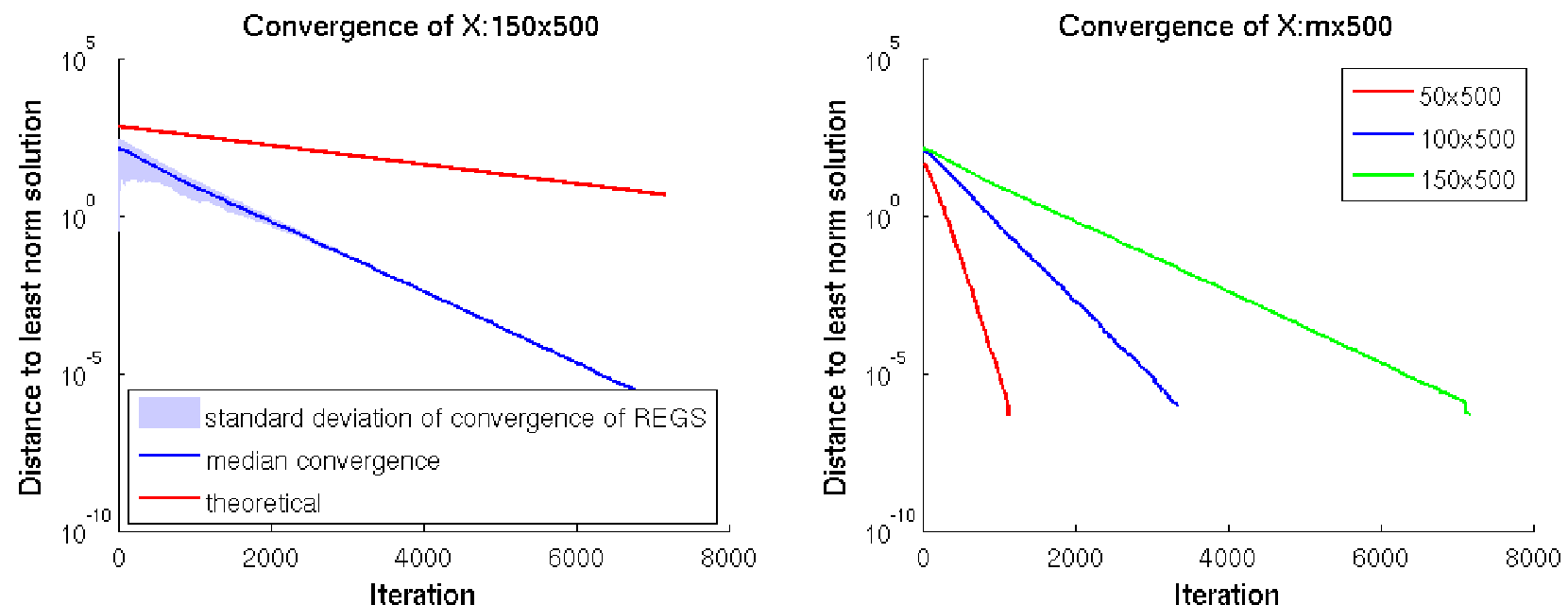}
        \caption{Left: $\ell_2$-error (log-scale) of REGS on a $150 \times 500$ matrix and its the theoretical bound. Right: Comparison of $\ell_2$-error (log-scale) of REGS for $m \times 500$ sized matrices with $m=50,100,150$.}\label{fig:vary_m}
\end{figure}

We also tested REGS on tomography problems using the Regularization toolbox by \citet{hansen1994regularization} (\texttt{http://www.imm.dtu.dk/$\sim$pcha/Regutools/}). For the 2D tomography problem $\bX\bbeta = \by$ with $\bX$ an $m \times n$ matrix where $n = dN^2$ and $m = N^2$, we use $N=20$ and $d=3$ for our experiments. Here, $\bX$ consists of samples of absorption along a random line on an $N \times N$ grid and $d$ is the oversampling factor. The results from this experiment are shown in Figure \ref{fig:fig2}.

\begin{figure}[!ht]
\includegraphics[natwidth=1024, natheight=384, scale=0.45]{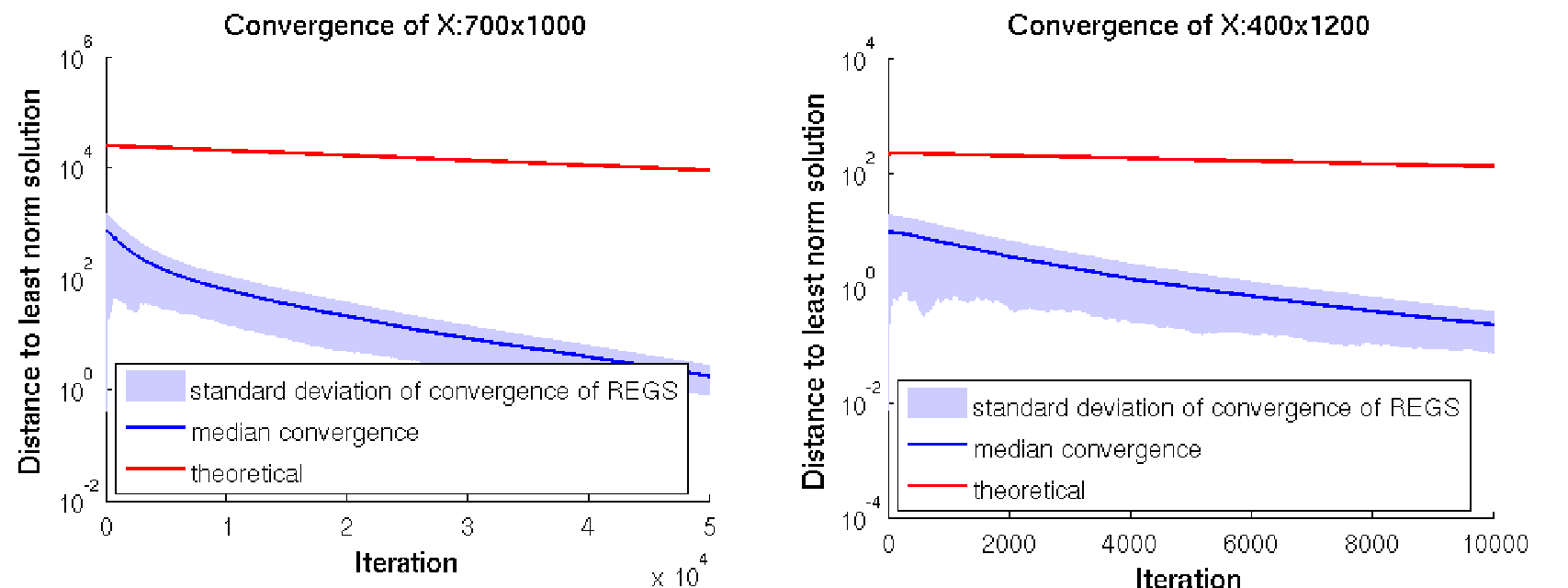}
\caption{Left: $\ell_2$-error (log-scale) of REGS on a $700 \times 1000$ matrix and its the theoretical bound.  Right: $\ell_2$-error (log-scale) of REGS on the tomography problem with a $400 \times 1200$ matrix.}\label{fig:fig2}
\end{figure}

We also compare the performance of all four algorithms (RK, REK, RGS, REGS) under the different settings discussed in this paper. Each line in each plot represents the median $\ell_2$-error at that iteration or CPU time over 50 trials using a stopping criterion of $10^{-6}$. For the underdetermined case, $\bX$ is a $50 \times 500$ Gaussian matrix and a $500 \times 50$ Gaussian matrix for the overdetermined cases. In the overdetermined, inconsistent case, we set $\by = \bX \bbeta + \bf{r}$ where $\bf{r} \in \text{null}(X^*)$ (computed in Matlab using the \texttt{null()} function). Figure \ref{fig:underincon}, Figure \ref{fig:overinconsis}, and Figure \ref{fig:overconsis} show the empirical results for the underdetermined, overdetermined inconsistent, and overdetermined consistent cases respectively.  Note we only plot the methods which actually converge to the desired solution in each case.  Looking at iterations to convergence, it seems that RK and RGS converge faster than their extended counterparts while REGS and REK converge to the desired solution at about the same rate.

\begin{figure}[!ht]
\includegraphics[natwidth=1024, natheight=384, scale=0.4]{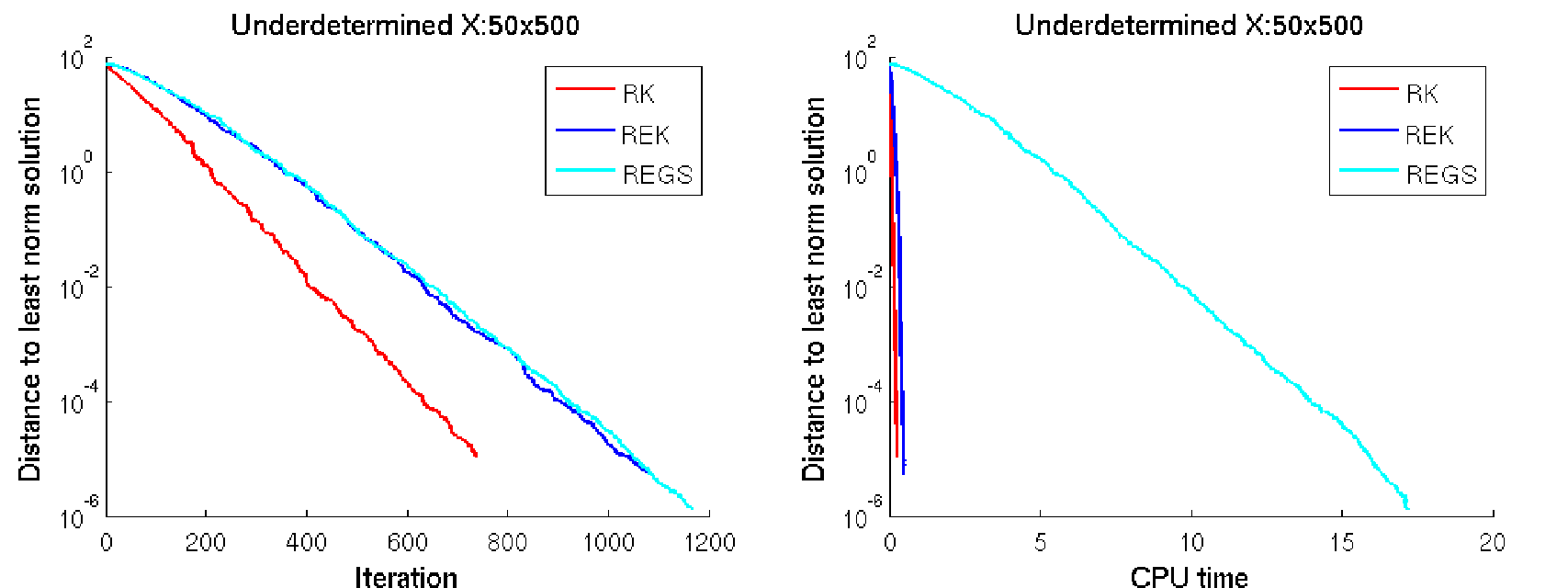}
	\caption{Comparison of median $\ell_2$-error (log-scale) of RK, REK, and REGS for an underdetermined system.}
	\label{fig:underincon}
\end{figure}

\begin{figure}[!ht]
	\includegraphics[natwidth=1001, natheight=384, scale=0.4]{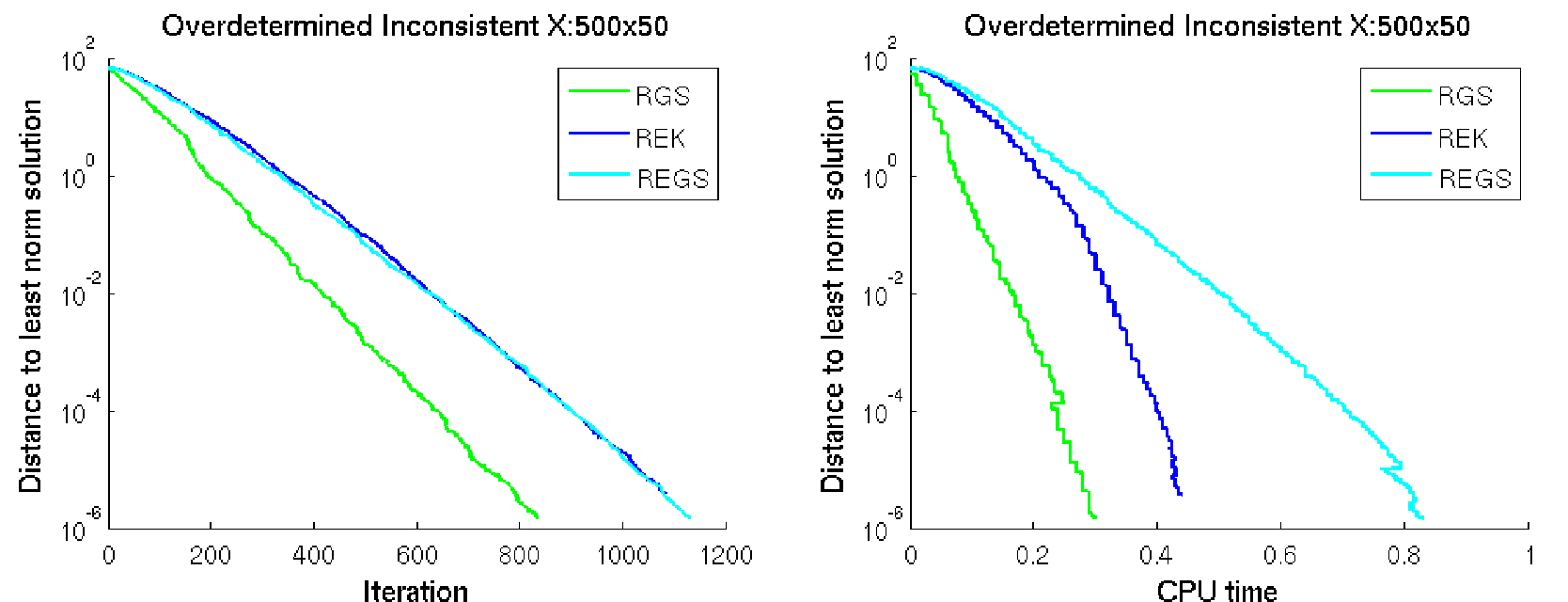}	
 \caption{Comparison of median $\ell_2$-error (log-scale) of RGS, REK, and REGS for an overdetermined, inconsistent system.}
	\label{fig:overinconsis}
\end{figure}
\begin{figure}[!ht]
 \includegraphics[natwidth=1024, natheight=384, scale=0.4]{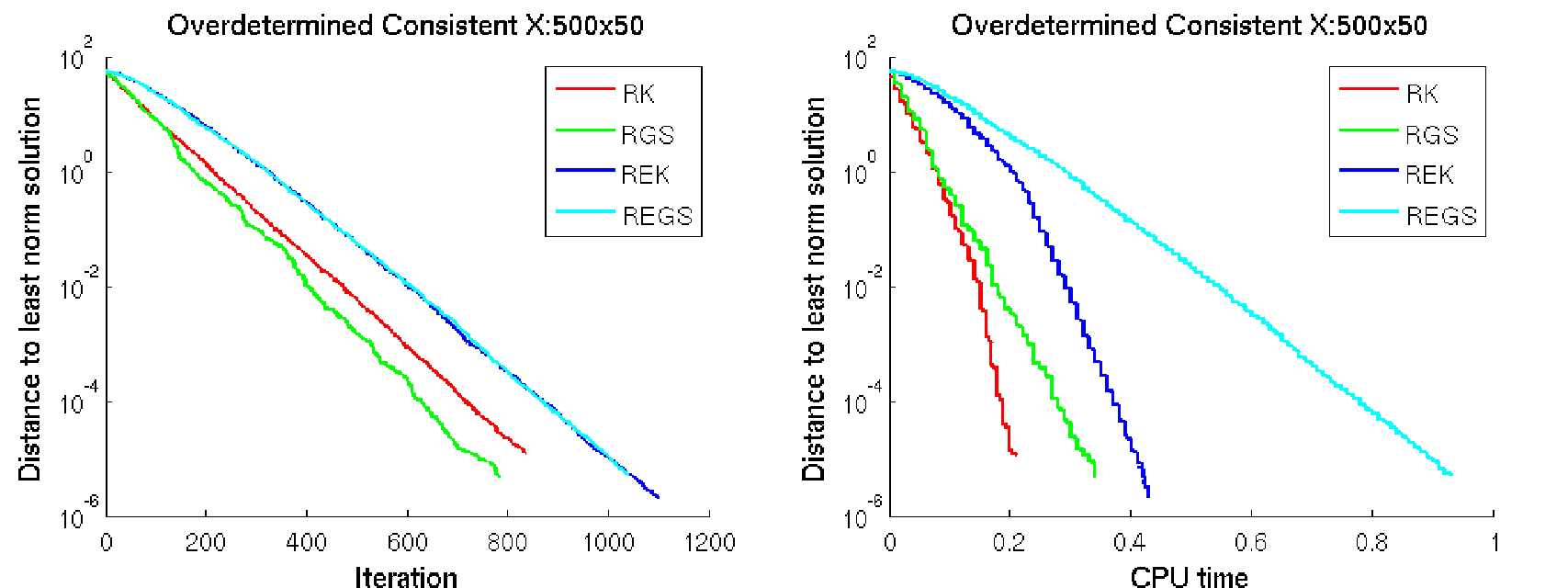}
	\caption{Comparison of median $\ell_2$-error (log-scale) of RK, RGS, REK, and REGS for an overdetermined, consistent system}
	\label{fig:overconsis}
\end{figure}

\section{Conclusion}\label{sec:conclude}
The Kaczmarz and Gauss-Seidel methods operate in two different spaces (i.e. row versus column space), but share many parallels.  In this paper we drew connections between these two methods, highlighting the similarities and differences in convergence analysis.  The approaches possess conflicting convergence properties; RK converges to the desired solution in the underdetermined case but not the inconsistent overdetermined setting, while RGS does the exact opposite.  The extended method REK in the Kaczmarz framework fixes this issue, converging to the solution in both scenarios.  Here, we present the REGS method, a natural extension of RGS, which completes the overall picture.  We hope that our unified analysis of all four methods will assist researchers working with these approaches.

\section*{Acknowledgments}
Needell was partially supported by NSF CAREER grant $\#1348721$, and the Alfred P. Sloan Fellowship. Ma was supported in part by AFOSR MURI grant FA9550-10-1-0569. Ramdas was supported in part by ONR MURI grant N000140911052. The authors would also like to thank the Institute of Pure and Applied Mathematics (IPAM) at which this collaboration started, and the reviewers of this manuscript for their thoughtful suggestions. We also thank the author of \cite{du2018refined} who pointed out the need to clarify what we have now written as Remark \ref{rem1}.

\bibliographystyle{agsm}
\bibliography{rk}

\end{document}